\documentclass[a4paper,reqno]{amsart}
\usepackage{graphicx}
\usepackage{xcolor} 
\usepackage{amsthm,amssymb,amsmath}
\usepackage{mathrsfs}
\usepackage{url}
\usepackage[russian, english]{babel}

\usepackage[utf8]{inputenc}

\newtheorem{thm}{Theorem}[section]
\newtheorem{lem}[thm]{Lemma}
\newtheorem{pro}[thm]{Proposition}

\theoremstyle{definition}
\newtheorem{defi}[thm]{Definition}
\theoremstyle{remark}
\newtheorem{exa}[thm]{Example}
\newtheorem{rema}[thm]{Remark}
\numberwithin{equation}{section}

\def\ep{\varepsilon}

\def\Q{\mathbb{Q}}
\def\Z{\mathbb{Z}}
\def\F{\mathbb{F}}
\def\A{\mathbb{A}}
\def\cuA{\mathcal{A}}
\def\cuB{\mathcal{B}}
\def\cuC{\mathcal{C}}
\def\cuQ{\mathcal{Q}}
\def\v{\mathrm{v}}

\newcommand{\mmid}{\mathrel{\|}}

\renewcommand{\H}{\mathrm{H}}
\newcommand{\Gm}{\mathbb{G}_m}

\DeclareMathOperator{\rnk}{rk}
\DeclareMathOperator{\Br}{Br}
\DeclareMathOperator{\Pic}{Pic}
\DeclareMathOperator{\inv}{inv}
\DeclareMathOperator{\id}{id}
\DeclareMathOperator{\rk}{rk}
\renewcommand{\Im}{\mathrm{Im}}

\begin{document}

\title{Quartic del Pezzo surfaces with a Brauer group of order $4$}

\author{Julian Lyczak}

\address{IST Austria\\
	Am Campus 1\\
	3400 Klosterneuburg\\
	Austria}
\email{jlyczak@ist.ac.at}

\author{Roman Sarapin}

\address{V.N.Karazin Kharkiv National University\\
	Svobody Square 4\\
	Kharkiv\\
	61022\\
	Ukraine}
\email{romansarapin13@gmail.com}

\begin{abstract}
We study arithmetic properties of del Pezzo surfaces of degree $4$ for which the Brauer group has the largest possible order using different fibrations into curves. We show that if such a surface admits a conic fibration, then it always has a rational point. We also answer a question of Várilly-Alvarado and Viray by showing that the Brauer groups these surfaces cannot be vertical with respect to any projection away from a plane. We conclude that the available techniques for proving existence of rational points or even Zariski density do not directly apply if there is no Brauer--Manin obstruction to the Hasse principle.

In passing we pick up the first examples of quartic del Pezzo surfaces with a Brauer group of order~$4$ for which the failure of the Hasse principle is explained by a Brauer--Manin obstruction.
\end{abstract}

\maketitle

\date{\today}

\maketitle

\thispagestyle{empty}
\setcounter{tocdepth}{1}

\section{Introduction}

The famous Hasse--Minkowski theorem gives precise conditions for a quadratic form to have a non-trivial zero. For the solubility of a system of two quadratic forms there are still many open questions. In this paper we restrict to the study of smooth intersections of two quadrics in five variables. Such a system describes a geometric object $X \subseteq \mathbb P^4$ called a \textit{del Pezzo surfaces of degree $4$}.

As in the Hasse--Minkowski theorem, a natural first step is to study the inclusion $X(\Q) \subseteq X(\A_\Q)$ of the rational points into the \textit{adelic points}. Since we are considering homogeneous equations we can identify  $X(\A_\Q)$ with $\prod_{p \leq \infty} X(\Q_p)$, that is, an adelic solution consists of local solutions in each completion of $\Q$. We say that a class of varieties satisfies the \textit{Hasse principle} if for each member of the family the existence of local solutions implies the existence of a global solution. In general the Hasse principle can fail. One such example was the quartic del Pezzo surface
$$
\begin{cases}
x^2-5y^2=uv;\\
x^2-5z^2=(u+v)(u+2v),
\end{cases}
$$
by Birch and Swinnerton-Dyer \cite{BSD}.

This failure can be explained by the Brauer--Manin obstruction, as introduced by Manin \cite{ManinICM}; any element $\cuQ$ in the Brauer group of $X$ determines an intermediate set
\[
X(\Q) \subseteq X(\A_\Q)^{\cuQ} \subseteq X(\A_\Q)
\]
which can be used in some cases to show that $X(\Q)=\emptyset$ even if there are adelic points on $X$. It has been conjectured by Colliot-Th\'el\`ene and Sansuc \cite{CTSansuc} that for varieties such as del Pezzo surfaces, the Brauer--Manin obstruction is the \textit{only obstruction to the Hasse principle}, that is, $X(\Q)$ is empty if and only if $X(\A_{\Q})^{\Br} := \bigcap_{\cuQ \in \Br X} X(\A_{\Q})^{\cuQ}$ is.

Pioneering work of Swinnerton-Dyer shows that for quartic del Pezzo surfaces the Brauer group $\Br X/\Br_0 X$ is isomorphic to $\left(\Z/2\Z \right)^i$ for $i \in\{0,1,2\}$ \cite{SDBrauerGroupCubicSurfaces}. Furthermore, the Brauer group can be explicitly computed from the equations \cite{BBFL}, \cite{VAV}.

The conjecture by Colliot-Th\'el\`ene and Sansuc was proven by Wittenberg \cite{WittenbergBook} for quartic del Pezzo surface with a trivial Brauer group subject to the Schinzel hypothesis and the conjectured finiteness of Tate--Shafarevich groups of elliptic curves. These techniques were then applied by Várilly-Alvarado and Viray for certain surfaces with a Brauer group of order $2$.

We will study the arithmetic of quartic del Pezzo surfaces with a Brauer group of order $4$, using fibration $X \dashrightarrow \mathbb P^1$ into curves. First we show that the existence of a commonly studied fibration implies the Hasse principle.

\begin{thm}[Thm.~\ref{thm:conicfibrations}]\label{thm:thm1}
Let $X$ be a quartic del Pezzo surface with $\# \Br X/\Br_0 X =4$ over a number field $K$. If $X$ admits a conic fibration $X \dashrightarrow \mathbb P^1$ then $X(K)\neq \emptyset$.
\end{thm}

The work of Várilly-Alvarado and Viray in the case of a smaller Brauer group, however uses maps $X \dashrightarrow \mathbb P^1$ obtained by embedding $X$ anticanonically and projecting away from a plane. They ask if a Brauer group of order $4$ can be ``vertical'' with respect to such maps. The second result of this paper is that this is not the case.

\begin{thm}[Thm.~\ref{thm:nonewpoints}]
Let $X \subseteq \mathbb P^4$ be an anticanonically embedded quartic del Pezzo surface over a number field $K$ with $\# \Br X/\Br_0 X =4$ and $X(K)=\emptyset$. Then $\Br X$ is not vertical with respect to a map $f \colon X \dashrightarrow \mathbb P^1$ obtained by projecting away from a plane.
\end{thm}

This shows that the techniques of \cite{CTSSD}, \cite{WittenbergBook} and \cite{VAV} cannot be directly applied to prove that the Brauer--Manin obstruction is the only one to the Hasse principle for quartic del Pezzo surfaces with a Brauer group of order $4$.

Surfaces with two independent classes in the Brauer group have rarely shown up in the literature. Jahnel and Schindler \cite{JS} prove that quartic del Pezzo surfaces with a Brauer group of order $2$ (and even those for which the conjecture is true) are Zariski dense in the moduli space of all quartic del Pezzo surfaces. In contrast, Mitankin and Salgado study a subfamily where infinitely many members have a Brauer group of order $4$, but all of these surfaces admit a conic fibrations and hence have a rational point in the light of Theorem~\ref{thm:thm1}.

We will restrict to a different subfamily in which each member has a Brauer group of order $4$, which we will show to be non-empty.

\begin{thm}[Thm.~\ref{thm:failureWA}, Thm.~\ref{thm:curlyABorC}]\label{thm:intro2}
Let $X$ be a quartic del Pezzo surface with $\#\Br X/\Br_0 X =4$ given by a system
\[
\begin{cases}
y^2-pz^2 = (A_1u+B_1v)(C_1u+D_1v);\\
z^2-pz^2 = (A_2u+B_2v)(C_2u+D_2v),
\end{cases}
\]
where $p$ is an odd prime. Then $X$ fails weak approximation and there is at most one element $\cuQ \in \Br X/\Br_0 X$ for which $X(\A_{\Q})^{\cuQ}=\emptyset$.
\end{thm}

Note the contrast with the result that if $X(\A_\Q)^{\Br} = \emptyset$ then there exists an $\cuQ \in \Br X$ such that $X(\A_{\Q})^{\cuQ}=\emptyset$ for any quartic del Pezzo surface \cite[Rem.~2 following Lem.~3.4]{CTPoonen}.

For every quartic del Pezzo surface $X$ with $\Br X/\Br_0 X$ has order $4$, we can write down two explicit generators $\cuA_X$ and $\cuB_X$, where $\cuA_X$ is uniquely defined and $\cuB_X$ is any other non-trivial element. We produce examples in which either generator obstructs the Hasse principle.

\begin{thm}
Consider the surfaces
$$
Y\colon\begin{cases}
y^2-13x^2=uv;\\
z^2-13x^2=(2u-13v)(u-6v),
\end{cases}
$$
and
$$
S \colon \begin{cases}
y^2-13x^2=uv;\\
z^2-13x^2=(u+v)(153u+179v).
\end{cases}
$$
\begin{enumerate}
\item[(a)] The surfaces $Y$ and $S$ are everywhere locally soluble.
\item[(b)] The Brauer groups $\Br Y/\Br_0 Y$ and $\Br S/\Br_0 S$ are both isomorphic to $(\mathbb Z/2\mathbb Z)^2$.
\item[(c)] We have $Y(\A_\Q)^{\cuA_Y} = \emptyset$ for $\cuA_Y = \left(13,\frac{u-6v}v\right) \in \Br Y$, and $S(\A_\Q)^{\cuB_S} = \emptyset$ for $\cuB_S = \left(13,\frac{y+z}u\right) \in \Br S$.
\end{enumerate}
\end{thm}

Using Theorem~\ref{thm:intro2} we see that $\cuA_Y$ and $\cuB_S$ are the unique elements in respectively $\Br Y$ and $\Br S$ with these properties.

To prove for a general quartic del Pezzo surface with a Brauer group of order $4$ that $X(\Q)\ne\emptyset$ when neither $\mathcal A_X$, $\mathcal B_X$ nor $\cuA_X+\cuB_X$ obstructs the Hasse principle is still open.

\subsection*{Acknowledgements}

This paper is a product of an internship of the second author at IST Austria under the supervision of the first author.

\section{Preliminaries}

Let us fix our terminology.

\begin{defi}
Let $k$ be a field. A $k$-scheme $X$ is called \textit{nice} if it is proper, geometrically integral and smooth over $k$. A \textit{surface} will be a nice $2$-dimensional $k$-scheme.
\end{defi}

All varieties under consideration in this paper will be nice $\mathbb Q$-surfaces. So by the properness we can identify $X(\A_\Q)$ with $\prod_{p\leq \infty} X(\Q_p)$ where we write $\Q_\infty := \mathbb R$. We have the diagonal embedding $X(\Q) \hookrightarrow X(\A_\Q)$.

We will say that a class $\mathcal S$ of varieties satisfies the \textit{Hasse principle} when $X(\Q)=\emptyset$ if and only if $X(\A_\Q) = \emptyset$ for all $X \in \mathcal S$. If $X(\Q)$ is dense in $X(\A_\Q)$ we say that X satisfies \textit{weak approximation}.

\subsection*{The Brauer--Manin obstruction}

The failure of the Hasse principle or weak approximation can be explained using the Brauer group. So let us consider $\Br X := \H^2(X_{\text{\'et}},\Gm)$ and its natural subgroup $\Br_0 X := \Im(\Br \Q \to \Br X)$. Since $X$ is nice we have an inclusion $\Br X \hookrightarrow \Br \kappa(X)$, and we will only be interested in elements $\mathcal Q \in \Br X$ of order $2$. Hence such elements are always represented by a quaternion algebra $(f,g)$ over the function field $\kappa(X)$ of $X$.

Recall the \textit{invariant map} $\inv_v \mathcal Q \colon X(\Q_v) \to \Q/\Z$ of an element $\mathcal Q=(f,g) \in \Br X[2]$ at a place $v$ \cite[Thm.~1.5.36]{Poonen}, which are constant for $\mathcal Q \in \Br_0 X$. At a point $P \in X(\Q)$ for which $f,g\in \mathcal O^\times_{X,P}$ the invariant $\inv_v \mathcal Q(P)$ coincides with the Hilbert symbol $(f(P),g(P)) \in \{\pm 1\}$ after identifying the groups $\{\pm 1\}$ and $\frac12\Z/\Z$. By abuse of terminology we will say that $\inv_v\mathcal Q$ is \textit{surjective} if $\# \Im(\inv_v \mathcal Q) = 2$.

For $\mathcal Q \in \Br X$ we can consider
\[
X(\A_\Q)^{\mathcal Q}  := \{(P_v)_v \in X(\A_\Q)\mid \sum_{v \leq \infty} \inv_v \mathcal Q(P_v) =0\}.
\]
By global reciprocity we have $X(\Q) \subseteq X(\A_\Q)^{\cuQ} \subseteq X(\A_\Q)$ or even $X(\Q) \subseteq X(\A_\Q)^{\Br} \subseteq X(\A_\Q)$ for $X(\A_\Q)^{\Br} := \bigcap X(\A_\Q)^{\cuQ}$.

If $X(\A_\Q) \ne \emptyset$, but $X(\A_\Q)^{\Br} = \emptyset$, then the Hasse principle fails for $X$ and we say there is a \textit{Brauer--Manin obstruction to the Hasse principle}. If $X(\A_\Q)^{\Br} \subsetneq X(\A_\Q)$, then $X$ cannot satisfy weak approximation, and there is a \textit{Brauer--Manin obstruction to weak approximation}.

Note that if an invariant map of $\mathcal Q$ is surjective, then there is a Brauer--Manin obstruction to weak approximation, but not necessarily to the Hasse principle.

\subsection*{Quartic del Pezzo surfaces}

We will be interested in a particular type of surface, namely del Pezzo surface of degree $4$. For a general treatise on general del Pezzo surfaces and their arithmetic the reader is referred to \cite{arithmeticdps}. We will give the following characterisation of quartic del Pezzo surfaces, which we will take as the definition.

\begin{defi}
A \textit{del Pezzo surface of degree $4$} is a surface $X \subseteq \mathbb P^4$, hence in particular smooth, given by two quadratic forms $Q=0=\tilde Q$.
\end{defi}

It is known that the Brauer group of a quartic del Pezzo surface modulo constants is either trivial, of order $2$ or isomorphic to the Klein four-group. Algorithms to determine this isomorphism class and explicit generators are given in both \cite{BBFL} and \cite{VAV}. Let us introduce some notation from the latter algorithm to determine if we are dealing with a Brauer group of order $4$.

Represent the quadratic forms $Q$ and $\tilde Q$ by symmetric matrices $M$ and $\tilde M$. For a point $T=(\kappa \colon \lambda) \in \mathbb P^1(\bar{\Q})$ we define the symmetric matrix $M_T = \kappa M + \lambda \tilde M$, with associated quadratic form $Q_T$ over $\kappa(T)$. Let $f$ be the discriminant of $M_T$, then $\mathscr{S}=V(f)\subseteq\mathbb{P}^1$ is the locus of the degenerate fibres. For the quadratic forms $Q_T$ of rank $4$ we will need the determinant $\ep_T$ of its restriction to a linear subspace of codimension $1$ on which it is nondegenerate. The invariant $\ep_T$ is well defined up to squares.

\begin{pro}[{\cite[\textsection 4.1]{VAV}}]\label{pro:Brauergrouporder4}
Let $X \subseteq \mathbb P^4$ be a quartic del Pezzo surface. Then $\#\Br X/\Br_0 X =4$ if and only if there exists an $\ep\not\in\Q^{\times, 2}$, and three degree 1 points $T_i \in \mathscr{S}$ such that each $Q_{T_i}$ has rank $4$ and satisfies $\ep\cdot\ep_{T_i}\in\Q^{\times, 2}$.
\end{pro}

\section{A Brauer group of order $4$}

From now on $X/\mathbb Q$ will be a quartic del Pezzo surface anticanonically embedded in $\mathbb P^4$. Such a surface is the intersection of two quadrics. We will be particularly interested in the $X$ for which the Brauer group modulo constants has order $4$. 

\begin{pro}\label{pro:gen_sys}
A quartic del Pezzo surface $X$ over $\mathbb Q$ satisfies $\Br X/\Br \mathbb Q \cong (\mathbb Z/2\mathbb Z)^2$ precisely if it can be given by a system of equations of the form
\begin{equation}\label{gen_sys}
\begin{cases}
d_0y^2-\ep x^2=a_0u^2+2b_0uv+c_0v^2;\\
d_1z^2-\ep x^2=a_1u^2+2b_1uv+c_1v^2,
\end{cases}
\end{equation}
such that
\begin{enumerate}
\item $\ep\not\in\Q^{\times, 2}$,
\item $d_0=b_0^2-a_0c_0$ and $d_1=b_1^2-a_1c_1$,
\item $\ep d_0d_1d_2\in \Q^{\times, 2}$ where $d_2:=(b_1-b_0)^2-(a_1-a_0)(c_1-c_0)$, and
\item the quadratic forms $a_iu^2+2b_iuv+c_iv^2$ do not have a common projective root.
\end{enumerate}
\end{pro}

\begin{proof}
Proposition~4.2 in \cite{VAV} and Proposition~\ref{pro:Brauergrouporder4} yield that after a linear change of variables the matrices of $Q_{T_i}$ have the following form
$$
M_{T_0}=
\begin{pmatrix}
a_0 & b_0 & 0 & 0 & 0\\
b_0 & c_0 & 0 & 0 & 0\\
0 & 0 & m_0 & 0 & 0\\
0 & 0 & 0 & n_0 & 0\\
0 & 0 & 0 & 0 & 0
\end{pmatrix},\ 
M_{T_1}=
\begin{pmatrix}
a_1 & b_1 & 0 & 0 & 0\\
b_1 & c_1 & 0 & 0 & 0\\
0 & 0 & m_1 & 0 & 0\\
0 & 0 & 0 & 0 & 0\\
0 & 0 & 0 & 0 & k_1
\end{pmatrix}
\ 
\text{ and }
\ 
M_{T_2}=
\begin{pmatrix}
a_2 & b_2 & 0 & 0 & 0\\
b_2 & c_2 & 0 & 0 & 0\\
0 & 0 & 0 & 0 & 0\\
0 & 0 & 0 & n_2 & 0\\
0 & 0 & 0 & 0 & k_2
\end{pmatrix}.
$$
We get that
$$
\ep_{T_0}=-d_0m_0 n_0,\quad\ep_{T_1}=-d_1m_1 k_1 \quad \text{ and } \quad \ep_{T_2}=-d_2n_2k_2,
$$
where $d_i:=b_i^2-a_ic_i$. Since the $M_{T_i}$ are linearly dependent we can assume $M_{T_2}=M_{T_1}-M_{T_0}$. Without loss of generality we can set $m_0=m_1$ equal to $\ep$. We can still change $n_0$ and $k_1$ independently up to squares, and the condition $\ep\cdot\ep_{T_i}\in\mathbb{Q}^{\times,2}$ allows us to take $n_0=-d_0$ and $k_1=-d_1$. The condition $\ep d_0d_1d_2\in\mathbb{Q}^{\times,2}$ is then automatically satisfied. The last condition is equivalent to the projective scheme defined by \eqref{gen_sys} is smooth.

From Proposition~\ref{pro:Brauergrouporder4} it is easy to check that any such system defines a surface with a Brauer group of order $4$.
\end{proof}

The algorithm in \cite{VAV} even allows us to represent the elements of the Brauer group by explicit quaternion algebras. Given the explicit element of the Brauer group we can now discuss whether the Brauer--Manin obstruction to Hasse principle is the only one. For simplicity we restrict to surfaces over the rational numbers, but the results and proofs in this and the last section equally apply to quartic del Pezzo surfaces with a Brauer group of order $4$ over any number fields. 

For quartic del Pezzo surfaces with a large Brauer group which admit conic fibrations there cannot be a Brauer--Manin obstruction to the Hasse principle.

\begin{thm}\label{thm:conicfibrations}
Let $X/\mathbb Q$ be a quartic del Pezzo surface with $\# \Br X/\Br_0 X =4$. If $X$ admits a conic fibration $X \dashrightarrow \mathbb P^1$ then $X(\mathbb Q)\neq \emptyset$.
\end{thm}

\begin{proof}
Since $X$ is a nice variety over a number field we conclude from the Hochschild--Serre spectral sequence \cite[Cor.~6.7.8]{Poonen} that $\Br X/\Br_0 X \cong \H^1(G_{\mathbb Q}, \Pic \bar X)$. We will need to understand the Galois action on the geometric Picard group.

For a general quartic del Pezzo surface the action of the absolute Galois group $G_{\mathbb Q}$ on $\Pic \bar X$ factors through a subgroup $H \subseteq W_5$. Here $W_5$ is the finite group Weyl group which acts on $\Pic \bar X$, see \cite[\textsection 1.5]{arithmeticdps}. Since $\Pic \bar X$ is torsion-free we conclude from the inflation-restriction sequence that $\H^1(G_{\mathbb Q}, \Pic \bar X) \cong \H^1(H, \Pic \bar X)$. For each of the 196 different subgroups of $W_5$ (up to conjugation) we can compute the first cohomology group directly and we find that only four subgroups yield a Brauer group of order $4$.

Again using the Hochschild--Serre spectral sequence we see that $\Pic X \hookrightarrow (\Pic \bar X)^{G_{\mathbb Q}}$ is injective. For the four relevant Galois actions on $\Pic \bar X$ we see that $(\Pic \bar X)^{G_{\mathbb Q}}$ has rank $1$ in three cases, and rank $2$ in one case.

In the case of rank $1$ we conclude that $\Pic X$ is free of rank $1$, since it contains the canonical class. If $X$ would admit a conic fibration then the class of the fibre should be a (rational) multiple of the canonical class, and have self-intersection $0$. Clearly a contradiction.

In the case of rank $2$ we find that the action of Galois permutes two exceptional curves, whose intersection pairing equals $1$. Hence their intersection point is defined over the base field.
\end{proof}

\begin{rema} More precisely, in the latter case there are multiple pairs of such intersecting lines which sum to the same class in $\Pic X$. One can conclude from this that among the quartic del Pezzo surfaces with a Brauer group of order $4$ the following statements are equivalent.
\begin{enumerate}
\item $\rk \Pic X =2$.
\item $X$ contains two conjugate intersecting lines.
\item $X$ admits a conic fibration.
\end{enumerate}
\end{rema}

For general maps $f\colon X\dashrightarrow \mathbb{P}^1$ the following notion will be very useful. We say that $\Br X$ is \textit{vertical} with respect to $f$ if $\Br X\subseteq f^*(\Br \kappa(\mathbb{P}^1))$.

For quartic del Pezzo surfaces with vertical Brauer groups there is a method to prove that $X(\A_\Q)^{\Br}\neq\emptyset$ 
implies $X(\Q)\neq\emptyset$ (see \cite{WittenbergBook} and \cite{CTSSD}). Várilly-Alvarado and Viray proved in \cite{VAV} that the Brauer group is vertical for certain quartic del Pezzo surfaces, and they concluded that the Brauer--Manin obstruction to Hasse principle is the only one for these surface of ``BSD-type''.

Next they considered quartic del Pezzo surface with a Brauer group of order $4$. First they noted that if $X(\Q)\neq \emptyset$, then $\Br X$ is vertical with respect to a specific map. Next they asked if $X(\Q)= \emptyset$ whether the Brauer group can be vertical, in particular for rational maps $f \colon X \dashrightarrow \mathbb P^1$ obtained from projecting away from a plane. They suggested the following approach.

Let $T_i$ be the three degree 1 points of $\mathscr{S}$, see Proposition~\ref{pro:Brauergrouporder4}. For $P_i\in V(Q_{T_i})(\Q)$, denote $M(P_0,P_1,P_2)$ the matrix whose $(i,j)$th entry is $\frac{\partial Q_{T_i}}{\partial x_j}(P_i)$. Then take a hyperplane $H\subseteq \mathbb{P}^4$ which does not pass through the vertices of $V(Q_{T_i})$, and define
$$
Y^{(X)}=\{(P_0,P_1,P_2)\in H^3\mid \rnk M(P_0,P_1,P_2)\le 2 \text{ and } Q_{T_i}(P_i)=0\}.
$$
There is an inclusion $X(\Q)\hookrightarrow Y^{(X)}(\Q)$, which is described in \cite{VAV}.

If we can find a rational point on $Y^{(X)}$, then we obtain the verticality of $\Br X/\Br\Q$ with respect to a certain map $f$. We already know that if $X(\Q)\neq \emptyset$, then $Y^{(X)}(\Q)\neq\emptyset$. The question is, whether there is a point in $Y^{(X)}(\Q)$ which does not come from a point in $X(\Q)$ (see \cite{VAV}, Question~6.3). Here we show that the answer is no.

\begin{defi}
We define a point $P=(u \colon v \colon x \colon y \colon z) \in \mathbb P^4(\Q)$ to be equivalent to the points $(u \colon v \colon \pm x \colon \pm y \colon \pm z)$. This induces an equivalence relation $\sim$ on $Y^{(X)}(\Q)$.
\end{defi}

In other words, two points $\mathcal P = (P_0,P_1,P_2)$ and $\mathcal P' = (P'_0,P'_1,P'_2)$ are equivalent precisely if the $P_i$ and $P'_i$ are the same projective point up to a possible change of sign in the last three coordinates.


Consider a point $\mathcal P \in Y^{(X)}(\Q)$. Although $\mathcal P$ need not lie the image of the map $X(\Q)\hookrightarrow Y^{(X)}(\Q)$, we now show that it is equivalent to a point that does.

\begin{thm}\label{thm:nonewpoints}
The composite map
$$
X(\Q)\hookrightarrow Y^{(X)}(\Q)\longrightarrow Y^{(X)}(\Q)/\sim
$$
is surjective.
\end{thm}

\begin{proof}
We will represent $X$ by a system of equations as in \eqref{gen_sys}. In addition, we can assume without loss of generality $a_0 \ne 0$ and hence after completing the square that $b_0=0$. This implies that $d_0 = -a_0c_0$ and we conclude that also $c_0 \ne 0$.

Consider the point $(P_0,P_1,P_2)\in Y^{(X)}(\Q)$ where $P_i=(u_i:v_i:x_i:y_i:z_i)$. Our goal is to prove that $(P_1,P_2,P_3)\sim (P,P,P)$ for some $P\in X(\Q)$. By the definition of $Y^{(X)}$ using $T_0=(1:0)$, $T_1=(0:1)$ and $T_2=(-1:1)$ as in the proof of Proposition~\ref{pro:gen_sys}, we have
$$
\mathrm{rank\,}
\begin{pmatrix}
2a_0u_0                & 2c_0v_0                & 2\ep x_0  & -2d_0y_0  & 0    \\
2a_1u_1+2b_1v_1        & 2b_1u_1+2c_1v_1        & 2\ep x_1  & 0      & -2d_1z_1\\
2(a_1-a_0)u_2+2b_1v_2  & 2b_1u_2+2(c_1-c_0)v_2  & 0            & 2d_0y_2   & -2d_1z_2
\end{pmatrix}
=2.
$$
Write $\ell_i \in\Q^5$ for the rows of the above matrix. The condition on the rank implies the existence of a non-trivial relation $\kappa\ell_0+\lambda \ell_1+\mu \ell_2=0$. After multiplying the coordinates of $P_0$, $P_1$ and $P_2$ by respectively $\kappa$, $-\lambda$ and $\mu$, we obtain $\ell_0-\ell_1+\ell_2=0$. This implies 
\begin{equation}\label{xyz}
x_0=x_1,\ y_0=y_2\ \text{ and }\ z_1=z_2,
\end{equation} 
and the first two columns give us the relations
\begin{equation}\label{subst}
u_0=\frac{1}{a_0}\Bigl(a_1u_1+b_1v_1-(a_1-a_0)u_2-b_1v_2\Bigr)
\quad \text{ and } \quad
v_0=\frac{1}{c_0}\Bigl(b_1u_1+c_1v_1-b_1u_2-(c_1-c_0)v_2\Bigr).
%
\end{equation}
Condition $Q_{T_i}(P_i)=0$ gives us
\begin{multline}\label{quadr}
(a_0u_0^2+c_0v_0^2)-(a_1u_1^2+2b_1u_1v_1+c_1v_1^2)+((a_1-a_0)u_2^2+2b_1u_2v_2+(c_1-c_0)v_2^2)=\\
(d_0y_0^2-\ep x_0^2)-(d_1z_1^2-\ep x_1^2)+(d_1z_2^2-d_0y_2^2)=0.
\end{multline}
Substituting \eqref{subst} into \eqref{quadr}, we obtain the following:
\begin{multline}\label{quad_form}
(a_0^2c_0+b_0^2a_0-a_1a_0c_0)(u_1-u_2)^2+2(a_1b_1c_0+b_1c_1a_0-b_1a_0c_0)(u_1-u_2)(v_1-v_2)+\\
(b_1^2c_0+c_1^2a_0-c_1a_0c_0)(v_1-v_2)^2=0.
\end{multline}
This is a quadratic form in $u_1-u_2$ and $v_1-v_2$ whose determinant equals 
$$
-a_0c_0(b_1^2-a_1c_1)(b_1^2-(a_1-a_0)(c_1-c_0))=d_0d_1d_2.
$$ 
From $\ep\notin\Q^{\times,2}$ and $\ep d_0d_1d_2\in\Q^{\times,2}$ we conclude $d_0d_1d_2\notin\Q^{\times,2}$, so 
\eqref{quad_form} implies $u_1-u_2=0=v_1-v_2$. Then \eqref{subst} implies that $u_0=u_1=u_2$ and $v_0=v_1=v_2$.

Finally, conditions \eqref{xyz} and $Q_{T_i}(P_i)=0$ give us $x_0=x_1=\pm x_2$, $y_0=y_2=\pm y_1$ and $z_1=z_2=\pm z_0$, 
which means that $P_i \in X(\Q)$ and $P_0 \sim P_1 \sim P_2$ which proves the proposition.
\end{proof}

We deduce that the answer to Question~6.3 in \cite{VAV} is no. We conclude that one cannot directly apply to machinery of \cite{CTSSD} and \cite{WittenbergBook} to prove there are rational points on quartic del Pezzo surfaces with a Brauer group of order $4$.

\section{A subfamily}

We will study the arithmetic of quartic del Pezzo surfaces with a Brauer group of order $4$. In \cite{MS} such surfaces appeared, but those all had a rational point for obvious geometrical reasons; on these surfaces there is a pair of intersecting lines which, as a pair, are fixed by the Galois action. We will restrict to a subfamily for which such behaviour does occur.

\begin{defi}
Let $\mathcal X$ be the set of isomorphism classes of quartic del Pezzo surfaces for which the two quadratic forms $Q_i=a_iu^2+b_iuv+c_iv^2$ split over $\mathbb Q$, and $\ep$ is an odd prime number $p$.
\end{defi}

\begin{pro}\label{pro:sys_subfam}
Any $X \in \mathcal X$ can be written as
\begin{equation}\label{sys}
\begin{cases}
y^2-px^2=Muv;\\
z^2-px^2=(Au+Bv)(Cu+Dv),
\end{cases}
\end{equation}
for $A,B,C,D,M,N\in\Z$ which satisfy
\begin{enumerate}
\item[\textbf{(C1)}] $(AD+BC-M)^2-4ABCD=pN^2$, and
\item[\textbf{(C2)}] $ NM(AD-BC)\neq 0$.
\end{enumerate}
\end{pro}

Condition (C1) is a restatement of $\ep d_0d_1d_2\in\Q^{\times,2}$ and (C2) is equivalent to $X$ being smooth. 

\begin{proof}
It is an immediate consequence of Proposition \ref{pro:gen_sys}.
\end{proof}

To study the members of this family which fail the Hasse principle we will first consider the question of local solubility. For places $v$ such that $p\in\Q_v^{\times,2}$ we have $(0\colon 0\colon 1\colon \sqrt{p}\colon \sqrt{p})\in X(\Q_v)$, so $X(\Q_v)\neq\emptyset$. The following proposition covers all but two of the other places.


\begin{pro}\label{pro:loc_sol_subfam}
Consider $X=X_{p,A,B,C,D,M,N}$. Let $v\not \in \{2,p\}$ be a place such that $p\notin\Q_v^{\times,2}$ and $v\nmid N$. Then $X(\Q_v)
\neq\emptyset$.
\end{pro}

\begin{proof}
Obviously $v\neq \infty$, so it corresponds to a prime number $q$. The Chevalley--Warning theorem yields that the system \eqref{sys} has a non-trivial solution $(u,v,x,y,z)\in\mathbb{F}_q^5$. Let $(u_0,v_0,x_0,y_0,z_0)$ be an arbitrary lift to $\Z_q^5$. We distinguish four cases.

\noindent\textbf{Case 1.} $q \nmid y_0,z_0$. In this case take 
$$
y_1=\sqrt{Mu_0v_0+px_0^2}\in\Z_q,\quad z_1=\sqrt{(Au_0+Bv_0)(Cu_0+Dv_0)+px_0^2}\in\Z_q,
$$ 
then $(u_0\colon v_0\colon x_0\colon y_1\colon z_1)\in X(\Q_q)$.

\noindent\textbf{Case 2.} $q \mid y_0$ and $q \nmid x_0,z_0$. Take 
$$
x_1=\sqrt{-\tfrac{M}{p}u_0v_0}\in\Z_q,\quad z_1=\sqrt{(Au_0+Bv_0)(Cu_0+Dv_0)+px_1^2}\in\Z_q,
$$ 
then $(u_0\colon v_0\colon x_1\colon 0\colon z_1)\in X(\Q_q)$.

\noindent\textbf{Case 3.} $q \mid x_0, y_0$ and $q \nmid z_0$. We deduce that $q \mid u_0v_0$. If $q \mid u_0$, take $(0\colon v_0\colon 0\colon 0\colon \sqrt{BDv_0^2})\in X(\Q_q)$. If on the other hand $q \mid v_0$, take $(u_0\colon 0\colon 0\colon 0\colon \sqrt{ACu_0^2})\in X(\Q_q)$.

\noindent\textbf{Case 4.} $q \mid y_0, z_0$. Subtracting the two equations in \eqref{sys} yields
$$
ACu^2+(AD+BC-M)uv+BDv^2\equiv 0\mod q,
$$
which can be rewritten to
$$
(2ACu+(AD+BC-M)v)^2 \equiv (AD+BC-M)^2v^2-4ABCDv^2 \equiv pN^2v^2 \mod q,
$$
and equivalently
$$
(2BDv+(AD+BC-M)u)^2 \equiv pN^2u^2 \mod q.
$$
From $\left(\frac{p}{q}\right)=-1$ and $q\nmid N$ we obtain $u\equiv v\equiv 0\mod q$, and \eqref{sys} gives us $x\equiv 0\mod q$. Contradicting the non-triviality of $(u,v,x,y,z)\in\F_q^5$.
\end{proof}

Most of our examples will satisfy $N=1$ so $X$ will be locally soluble everywhere away from $2$ and $p$.

The algorithm in \cite[\textsection 4.1]{VAV} allows to write down the explicit elements of $\Br X/\Br \Q$. Choosing $P_{T_0}=(0\colon 1\colon 0\colon 0\colon 0)\in Q_{T_0}(\Q)$, $P_{T_1}=(-B\colon A\colon 0\colon 0\colon 0)\in Q_{T_1}(\Q)$ and $P_{T_2}=(0\colon 0\colon 0\colon 1\colon 1)\in Q_{T_2}(\Q)$, we obtain
\begin{equation}\label{Br_subfam}
\Br X/\Br \Q=\left\{\id, \left(p,\frac{u}{Au+Bv}\right), \left(p,\frac{z-y}{u}\right), \left(p,\frac{Au+Bv}{z-y}
\right)\right\}.
\end{equation}

\begin{defi}\label{defi:ABC}
For any $X\in\mathcal{X}$ given by \eqref{sys}, we will write $\cuA=\left(p,\frac{u}{Au+Bv}\right)$, $\cuB=\left(p,\frac{z-y}{u}\right)$ and $\cuC=\left(p,\frac{Au+Bv}{z-y}\right)$ and assume the dependency on $X$ to be understood.
\end{defi}

\begin{rema}
The above notation actually depends on the representation \eqref{sys} we chose for the equivalence class $X\in\mathcal{X}$. One can see that the change of variables $\tilde{u}=\frac{Au+Bv}{AD-BC}$, $\tilde{v}=\frac{Cu+Dv}{AD-BC}$, $\tilde{y}=z$, $\tilde{z}=y$ provides another representation of $X$ in the form \eqref{sys}, but with possibly different coefficients, and $\cuB$ for the new representation is exactly $\cuC$ for the old one. 
However, one can show that $\cuA$ does not depend on the representation of $X$.


\end{rema}

\begin{rema}
One can see that the classes $\cuA$ and $\cuB$ are also represented by
\begin{equation}\label{Br_gr_rewr}
\cuA=\left(p,\frac{Mv}{Au+Bv}\right) \quad \text{ and } \quad
\cuB=\left(p,\frac{AC(z+y)}{u}\right).
\end{equation}
\end{rema}

\section{Simultaneous obstructions to the Hasse principle}

We use the notation from the previous section. In particular, $X$ is a quartic del Pezzo surface given by explicit equations as in Proposition~\ref{pro:sys_subfam}. We then see from \eqref{Br_subfam} and Definition~\ref{defi:ABC} that
\[
\Br X/\Br \mathbb Q = \{1,\mathcal A, \mathcal B, \mathcal C\}.
\]
We will concern ourselves with the Brauer--Manin obstruction to weak approximation and the Hasse principle. The first results is on weak approximation.

\begin{thm}\label{thm:failureWA}
Any $X \in \mathcal X$ fails weak approximation.
\end{thm}

The second result is about the Hasse principle and seems related to the following fact by Colliot-Th\'el\`ene and Poonen \cite[Rem.~2 following Lem.~3.4]{CTPoonen} on quartic del Pezzo surfaces: suppose that $X(\mathbb A_\Q)^{\Br} = \emptyset$ then there exists an element $\mathcal Q \in \Br X$ such that
 $X(\mathbb A_\Q)^{\mathcal Q} = \emptyset$. However, neither result implies the other.

\begin{thm}\label{thm:curlyABorC}
Only one of $\mathcal A$, $\mathcal B$ and $\mathcal C$ can give an obstruction to the Hasse principle on $X \in \mathcal X$.
\end{thm}

Theorem~\ref{thm:failureWA} and Theorem~\ref{thm:curlyABorC} are direct consequences of the following proposition.

\begin{pro}\label{prop:surjectiveinvariantmap}
Consider an $X\in\mathcal{X}$. Assume $X(\Q_p)\neq\emptyset$, then the invariant map at $p$ of $\cuA$, $\cuB$ or $\cuC$ is surjective.
\end{pro}

We will need the following lemmas.

\begin{lem}\label{lem:quadres}
Let $p\equiv 1\mod 4$ be a prime. Consider the set $S(a,b)=\{a+by\mid y\in\F_p^{\times,2}\}$ for $a,b\in\F_p^\times$. Denote $S(a,b)_\zeta=\{x\in S(a,b)\mid \left( \frac xp\right)=\zeta\}$. 
\begin{enumerate}
\item[(a)] If $a,b\in\F_p^{\times,2}$, then $|S(a,b)_0|=1$, $|S(a,b)_1|=\dfrac{p-5}{4}$ and $|S(a,b)_{-1}|=\dfrac{p-1}{4}$.
\item[(b)] If $a\in\F_p^{\times,2}$, $b\notin\F_p^{2}$, then $|S(a,b)_0|=0$, $|S(a,b)_1|=\dfrac{p-1}{4}$ and $|S(a,b)_{-1}|=\dfrac{p-1}{4}$.
\end{enumerate}
In particular, for $a\in\F_p^{\times,2}$ we conclude that $S(a,b)$ contains a non-square.
\end{lem}

\begin{proof}
Both parts follow from the identity $\sum_{x\in\F_p} (\frac{a+bx^2}{p})=-(\frac{b}{p})$ where $a,b\in\F_p^\times$.
\end{proof}

\begin{lem}\label{lem:quadres_table}
Let $p\equiv 1\mod 4$ be a prime and fix $a,b,c,d\in\F_p^{\times,2}$. Then there exists an $y_0\in\F_p^\times$ such that either
\begin{enumerate}
\item[(i)] $a+by_0^2,\ c+dy_0^2\in\F_p^{\times}\setminus\F_p^{\times,2}$, or
\item[(ii)] $a+by_0^2\in \F_p^{\times}\setminus\F_p^{\times,2}$ and $c+dy_0^2=0$.
\end{enumerate}
\end{lem}

\begin{proof}
Apply the statement (a) of Lemma~\ref{lem:quadres} for $S(a,b)$ and $S(c,d)$.
\end{proof}

\begin{proof}[Proof of Proposition~\ref{prop:surjectiveinvariantmap}]
We will first make a few reductions.

First suppose that $p \mid M, A, C$. Then the system \eqref{sys} is equivalent to the one with $M$, $A$ and $B$ divided by $p$. Also if $p^2 \mid M, A, C$ then we obtain an equivalent system by dividing these three coefficients by $p^2$. This shows that we can simply to the cases where
\begin{equation}\label{gcd_cond_1}
 p \nmid \gcd(A,C,M) \text{ and } p \nmid \gcd(B,D,M),
\end{equation}
and
\begin{equation}\label{gcd_cond_2}
p^2\nmid\gcd(A,B,M) \text{ and } p^2\nmid \gcd(C,D,M).
\end{equation}

Consider the case $p \equiv 3 \mod 4$. For a point $P=(u\colon v \colon x \colon y \colon z) \in X(\mathbb Q_p)$ we compute the invariant map at $P$ and $P'=(u\colon v \colon x \colon -y \colon -z)$. We see that
\[
\inv_p \cuB(P) \ne \inv_p \cuB(P')
\]
and this proves the proposition in this case.

Now we can assume 
\begin{equation}\label{p_1_mod4}
p\equiv 1\mod 4.
\end{equation}
Condition (C1) implies $(p,ABCD)_p=1$. Consider the case $(p,AC)_p=(p,BD)_p=-1$. For a point $P=(u\colon v\colon x\colon y\colon z)\in X(\Q_p)$, we define $P'=(u\colon v\colon x\colon -y\colon z)\in X(\Q_p)$. One can see that $\inv_p\cuB(P)\neq\inv_p\cuB(P')$ using \eqref{Br_gr_rewr}, and this proves the proposition in this case.

So now we can assume
\begin{equation}\label{hilbert_cond}
(p,AC)_p=(p,BD)_p=1.
\end{equation}

For equations satisfying (C1) and (C2), and the additional assumptions \eqref{gcd_cond_1} up to \eqref{hilbert_cond} we consider different cases depending on the valuation of $M$ and $m = \max\{v_p(A),v_p(B),v_p(C),v_p(D)\}$. Without loss of generality we can assume that $p^m \mmid A$ and $p^{m+1} \nmid B,C,D$.

In the first two cases we show that the invariant map of $\cuB$ at $p$ is surjective.
\begin{enumerate}
\item[\textbf{Case 1.}] $\v_p(M)=0$ and $m\geq 1$.
\item[\textbf{Case 2.}] $\v_p(M)=1$ and $m=1$.
\end{enumerate}
 In the next two cases we show that the invariant map of $\cuA$ at $p$ is surjective.
\begin{enumerate}
\item[\textbf{Case 3.}] $\v_p(M)=0$ and $m=0$.
\item[\textbf{Case 4.}] $\v_p(M)$ is odd and $m=0$.
\end{enumerate}
In the last four cases we can make a change of variables to reduce to one of the previous cases.
\begin{enumerate}
\item[\textbf{Case 5.}] $\v_p(M)=1$ and $m\geq 2$.
\item[\textbf{Case 6.}] $\v_p(M) \geq 3$ is odd and $m \geq 1$.
\item[\textbf{Case 7.}] $\v_p(M) \geq 2$ is even and $m=0$.
\item[\textbf{Case 8.}] $\v_p(M) \geq 2$ is even and $m\geq 1$.
\end{enumerate}

\noindent
\textbf{Surjectivity for $\cuB$}.
To prove $\inv_p\mathcal B$ is surjective we will find points $P_i=(u_i\colon v_i\colon x_i\colon y_i\colon z_i)\in X(\Q_p)$ such that 
$$
\left(\frac{u_1(z_1-y_1)}{p}\right)\left(\frac{u_2(z_2-y_2)}{p}\right)=-1,
$$ 
since this implies $\inv_p\cuB(P_1)+\inv_p\cuB(P_2)=\frac{1}{2}$.

\noindent\textbf{Case 1.} Since $p \mid A$, condition (C1) implies that 
$
BC\equiv M\not\equiv 0\mod p.
$
Hence $\v_p(B)=\v_p(C)=0$.

\textbf{Case 1a.} If $\v_p(D)=0$ then \eqref{p_1_mod4} and \eqref{hilbert_cond} imply $-BD\in\Z_p^{\times, 2}$.
Choose $r\in\Z_p^{\times}$ such that $r\sqrt{-BD}\notin\Z_p^{\times, 2}$. Consider $y_2=\frac{1}{2}\left(\frac{BD}{r}-r
\right)$ and note that
\begin{equation*}
-\frac{DM}{C}\equiv -BD\mod p
\text{\quad and\quad}
\left(\frac{A}{M}y_2^2+B\right)\left(\frac{C}{M}y_2^2+D\right)\equiv y_2^2 + BD\equiv \frac{1}{4}\left(r+\frac{BD}{r}\right)^2 \not\equiv 0\mod p.
\end{equation*}
Then take 
$$
P_1=\left(-\frac DC \colon 1 \colon 0\colon \sqrt{-\dfrac{DM}C}\colon 0\right) \; \text{ and } \; P_2=\left(\frac{y_2^2}{M}\colon 1\colon 0\colon y_2\colon \sqrt{\left(\frac{A}{M}y_2^2+B\right)\left(\frac{C}{M}y_2^2+D\right)}\right)\in X(\Q_p).
$$

\textbf{Case 1b.} If $\v_p(D)\geq 1$ then take $y_1\in\Z_p^{\times, 2}$ and $y_2\notin \Z_p^{\times, 2}$.
Then
\begin{equation}
\left(\frac{A}{M}y_i^2+B\right)\left(\frac{C}{M}y_i^2+D\right)\equiv \frac{BC}{M}y_i^2\equiv y_i^2\mod p,
\end{equation}
so we can choose $z_i\in\Z_p$ such that $z_i\equiv -y_i\mod p$ and $z_i^2=\left(\frac{A}{M}y_i^2+B\right)\left(\frac{C}{M}y_i^2+D\right)$.
Now take 
$$
P_i=\left(1\colon \frac{y_i^2}{M}\colon 0\colon y_i\colon z_i\right)\in X(\Q_p).
$$

\noindent\textbf{Case 2.} In this case we have $p \mmid M$ and $p \mmid A$ and $p^2 \nmid B,C,D$. From \eqref{gcd_cond_1} we get $\v_p(C)=0$, so condition (C1) implies $p \mid B$. From $p^2 \nmid B$, \eqref{gcd_cond_1} and (C1) we conclude in order that $\v_p(B)=1$, $\v_p(D)=0$ and $\v_p(N)\ge 1$.

Introduce $A=pA'$, $B=pB'$, $M=pM'$ and $N=pN'$. Then (C1) becomes $(A'D+B'C-M')^2-4A'B'CD=pN'^2$ and we conclude $\v_p(A'D+B'C-M)=0$. 
Our system is equivalent to
$$
\begin{cases}
-x^2+\dfrac{y^2}{p}=M'uv;\\
\dfrac{4A'C}{p}(z^2-y^2)=\bigl(2A'Cu+(B'C+A'D-M')v\bigr)^2-pN'^2v^2.
\end{cases}
$$
\indent\textbf{Case 2a.} Suppose that $M'\cdot\frac{B'C+A'D-M'}{2A'C}\in\Z_p^{\times 2}$. 
Choose $r_1\in\Z_p^{\times 2}$ and $r_2\not\in\Z_p^{\times 2}$ and define 
$$
y_i=\frac{1}{2}\left(\frac{A'C}{r_i}-r_i\right)N',\quad z_i=\frac{1}{2}\left(\frac{A'C}{r_i}+r_i\right)N'\quad \text{ and } \quad x_i=\sqrt{M'(B'C+A'D-M')\cdot 2A'C+py_i^2}.
$$ 
Now we can just take
$
P_i=(-(B'C+A'D-M')\colon 2A'C\colon x_i\colon py_i\colon pz_i)\in X(\Q_p).
$

\textbf{Case 2b.} Now suppose that $M'\cdot\frac{B'C+A'D-M'}{2A'C}\notin\Z_p^{\times 2}$. Let $(u:v:x:y:z)\in X(\Q_p)$ be represented by a primitive tuple in $\Z_p$. Obviously, $\v_p(y),\v_p(z)\ge 1$ so write $y=py_1$ and $z=pz_1$. Then 
$$
(2A'Cu+(B'C+A'D-M')v)^2-pN'^2v^2=4A'Cp(z_1^2-y_1^2),
$$ 
hence $u\equiv -\frac{B'C+A'D-M'}{2A'C} v\mod p$. This turns the first equation into $M'\cdot\frac{B'C+A'D-M'}{2A'C} 
v^2\equiv x^2\mod p$. It is possible only if $u\equiv v\equiv x\equiv 0\mod p$. This contradicts the assumption $X(\Q_p) \ne \emptyset$.
\bigskip

\noindent\textbf{Surjectivity for $\cuA$.}
In Cases 3 and 4 we apply the same technique for proving that $\inv_p \mathcal A$ is surjective.

\noindent\textbf{Case 3.} By assumption $p \nmid MABCD$. Since $(p,AC)_p=(p,BD)_p=1$, we have $AC,BD\in \Z_p^{\times, 2}$.

\textbf{Case 3a.} If $ABM\notin \Z_p^{\times, 2}$ then $\mathcal A$ has different invariants at the points
$$
P_1=(1\colon 0\colon 0\colon 0\colon \sqrt{AC})\in X(\Q_p)\quad \text{ and } \quad P_2=(0\colon 1\colon 0\colon 0\colon \sqrt{BD})\in X(\Q_p).
$$
as one can see using \eqref{Br_gr_rewr}.

\textbf{Case 3b.} Now let $ABM\in\Z_p^{\times, 2}$. It follows from Lemma \eqref{lem:quadres_table} that there exists a $y_0\in \Z_p$ such that $1+\frac{B}{MA}y_0^2$, $\frac{C}{A}+\frac{D}{MA}y_0^2\not\in\Z_p^{\times, 2}$, or $\frac{C}{A}+\frac{D}{MA}y_0^2 \equiv 0 \mod p$ and $1+\frac{B}{MA}y_0^2\not\in\Z_p^{\times, 2}$. In the first case we consider
$$
P_1=(1\colon 0\colon 0\colon 0\colon \sqrt{AC})\quad \text{ and } \quad
P_2=\left(1\colon \frac{y_0^2}{M}\colon 0\colon y_0\colon A\sqrt{\left(1+\frac{B}{MA}y_0^2\right)\left(\frac{C}{A}+\frac{D}{MA}y_0^2\right)}\right)
\in X(\Q_p).
$$
In the second case we choose a $y_1\in\Z_p$ such that $y_1\equiv y_0\mod p$ and $\frac{C}{A}+\frac{D}{MA}y_1^2=0$. We then work with
$$
P_1=(1\colon 0\colon 0\colon 0\colon \sqrt{AC})\quad \text{ and } \quad  P_2=\left(1\colon \frac{y_1^2}{M}\colon 0\colon y_1\colon 0\right)\in X(\Q_p).
$$

\noindent\textbf{Case 4.} Let $M=p^{2k+1}M'$ with $k\geq 0$ and $M' \in \mathbb Z_p^\times$. From (C1) we obtain that $AD\equiv BC\mod p$. Lemma~\ref{lem:quadres}(b) for $S\left(1,\frac{B}{M'A}\right)$ implies the existence of an $x_0\in\Z_p^\times$ such that $1-\frac{B}{M'A}x_0^2\in\Z_p^\times\setminus\Z_p^{\times 2}$. Then 
$$
1-\frac{D}{M'C}
x_0^2\equiv 1-\frac{B}{M'A}x_0^2\mod p,
$$ 
so we can take 
$$
P_1=(1:0:0:0:\sqrt{AC})\, \text{ and }\, P_2=\left(1:-\frac{x_0^2}{M'}:p^kx_0:0:\sqrt{AC\left(1-\frac{Bx_0^2}{M'A}\right)\left(1-\frac{Dx_0^2}{M'C}\right)+p^{2k+1}x_0^2}\right).
$$

\noindent\textbf{Case 5.} We have $p\mmid M$ and $p^2 \mid A$. Conditions (C1) and \eqref{gcd_cond_1} imply $p\nmid C$, $p \mid B$, $p \nmid D$ and $p \mid N$. Then \eqref{sys} is isomorphic to the system with coefficients $(p^{-2}A,p^{-1}B,C,pD,p^{-1}M,p^{-1}N)$ under the morphism $(u\colon v \colon x \colon y \colon z) \mapsto (pu\colon v \colon x\colon y \colon z)$. For the new system we have $\v_p(M)=0$ and $m\geq 1$, which was considered in Case 1.

\noindent\textbf{Case 6.} From $p^{2k+1} \mmid M$ and $p\mid A$ we get, as in the previous case, $p \mid B,N$ and $p \nmid C,D$. Considering (C1) modulo $p^{2k+1}$ we get $(AD-BC)^2 \equiv pN^2 \mod p^{2k+1}$. From this we conclude that $p^{k+1} \mid AD-BC$ and $p^k \mid N$. In particular, $p^2 \mid A$ if and only if $p^2 \mid B$. So we conclude $p\mmid A,B$ from \eqref{gcd_cond_2}. This implies $2\v_p(AD-BC+M) = \v_p(pN^2-4ABCD) = 2$ and hence $\v_p(AD+BC)=1$. From this we deduce that
\[
2\v_p(AD-BC)=\v_p(pN^2 +2(AD+BC)M-M^2) = 2k+2.
\]
Hence $p^{k+1} \mmid AD-BC$.

Our system of equations is equivalent to \eqref{sys} with the coefficients
\[
(p^{-2k}MD,-p^{-2k-1}MB,-C,p^{-1}A,p^{-2k-1}(AD-BC)^2)
\]
under the morphism $(u\colon v \colon x \colon y \colon z) \mapsto (\frac{Au+Bv}{AD-BC} \colon p\frac{Cu+Dv}{AD-BC} \colon p^{-k}x \colon p^{-k}y \colon p^{-k}z)$. This new system satisfied $\v_p(M)=1$ and $m=1$, which was discussed in Case 2.

\noindent\textbf{Case 7.} Now we know that $p^{2k} \mmid M$ and $p\nmid ABCD$. Condition (C1) implies $AD\equiv BC\mod p$, so $\v_p(AD+BC)=0$. From $\v_p(M)=2k$ and (C1) we obtain that $\v_p(AD-BC)=k$. Now consider the coefficients
\[
(p^{-2k}MD,-p^{-2k}MB,-C,A,p^{-2k}(AD-BC)^2)
\]
with variables $(\frac{Au+Bv}{AD-BC}\colon \frac{Cu+Dv}{AD-BC}\colon p^{-k}x\colon p^{-k}y\colon p^{-k}z)$ for the system \eqref{sys}. This is precisely Case 3.

\noindent\textbf{Case 8.} Now we assume $p^{2k} \mmid M$ and $m=1$. Without loss of generality $p \mmid A$ and as in the previous cases we can prove $p \mmid B$, $p\nmid CD$, $\v_p(AD+BC)=1$ and $\v_p(AD-BC)\ge k+1$. The system \eqref{sys} with coefficients $(p^{-2k}MD,-p^{-2k-1}MB,-C,p^{-1}A,p^{-2k-1}(AD-BC)^2)$ with variables $(\frac{Au+Bv}{AD-BC}\colon p\frac{Cu+Dv}{AD-BC}\colon p^{-k}x\colon p^{-k}y\colon p^{-k}z)$ is isomorphic to our system. This new system satisfies $\v_p(M)$ is odd and $m=0$, which is precisely Case~4.
\end{proof}

\section{Explicit examples}

All known examples of Brauer--Manin obstructions for quartic del Pezzo surfaces occur on surfaces with a Brauer group modulo constants of order $2$. The first such obstruction was described by Birch and Swinnerton-Dyer in \cite{BSD}. Work by Jahnel and Schindler \cite{JS} showed that the locus of quartic del Pezzo surfaces with a Brauer group of order $2$ for which the Brauer--Manin obstruction is the only one, is dense in the moduli space of all del Pezzo surfaces of degree $4$.

Mitankin and Salgado \cite{MS} restricted to the subfamily in which a positive proportion has a Brauer group of order $4$. Such surfaces can be written as
$$
\begin{cases}
x_0x_1-x_2x_3=0;\\
a_0x_0^2+a_1x_1^2+a_2x_2^2+a_3x_3^2+a_4x_4^2=0,
\end{cases}
$$
where $a_0a_1,a_2a_3,-a_0a_2\in\Q^{\times,2}$ and $-a_0a_4(a_0a_1-a_2a_3)\notin\Q^{\times,2}$. The Hasse principle, however, trivially holds for such surfaces, because of the rational point $(1:0:\sqrt{-a_0/a_2}:0:0)$.

In this section we will exhibit families of quartic del Pezzo surfaces with a Brauer group of order $4$ with a Brauer--Manin obstruction to the Hasse principle.

\subsection{Obstructions coming from $\mathcal A$}

Let us define our first subfamily. Consider a prime number $p\equiv 1\mod 4$ and integers $a,b$ such that $ab=p-1$. The surface $Y_{p,a,b}$ is given by the system
$$
\begin{cases}
y^2-px^2=uv;\\
z^2-px^2=(au-pv)(u-bv).
\end{cases}
$$
Obviously, $Y_{p,a,b}=X_{p,a,-p,1,-b,1,1}$, and one can check that conditions (C1) and (C2) are satisfied with $N=1$.

For such surfaces we have $\cuA=\left(p,\frac{u}{au-pv}\right)$, $\cuB=\left(p,\frac{z-y}{u}\right)$ and $\cuC=\left(p,\frac{au-pv}{z-y}\right)$.

\begin{pro}\label{pro:ex1_inv_B}
Consider the surface $Y=Y_{p,a,b}$.
\begin{enumerate}
\item[(a)] The variety $Y$ is everywhere locally soluble.
\item[(b)] The invariant map of $\inv_p\cuB \colon Y(\mathbb A_\Q) \to \frac12\Z/\Z$ is surjective.
\end{enumerate}
\end{pro}

\begin{proof}
Local solubility at all places $v\not\in \{2,p\}$ has already been proven in Proposition \ref{pro:loc_sol_subfam}. Local solubility at $p$ and the surjectivity of the invariant map of $\cuB$ at $p$ follow from Case 1 in the proof of Proposition~\ref{prop:surjectiveinvariantmap}. For local solubility at $2$ one can simply consider all possibilities of $a$, $b$ and $p$ modulo $8$.
\end{proof}

So we might expect an obstruction coming from $\cuA$.

\begin{pro}\label{pro:ex1_inv_A}
For all $(s_v)_v\in Y_{p,a,b}(\A_\Q)$ we have: $\inv_v \cuA(s_v)=0$ for $v\neq p$, and:
$$
\inv_p \cuA(s_p)=\begin{cases}
0 & \text{ if }\left(\dfrac{a}{p}\right)=1;\\
\dfrac{1}{2} & \text{ if }\left(\dfrac{a}{p}\right)=-1.
\end{cases}
$$
\end{pro}

\begin{proof}
One can prove this by direct computation. However, we will apply results from \cite{Beffeval} and \cite{CTS13} to conclude that $\inv_v$ is constant for $v \ne p$ and compute it at a single point in $X(\Q_v)$ to conclude $\inv_v$ is identically zero.

Finally, consider the case $v=p$. Let $s_p=(u\colon v\colon x\colon y\colon z)$ with $(u,v,x,y,z)$ a primitive tuple in $\Z^5_p$. Suppose $p \mid u$, then immediately $p \mid y,z$. Denote $u=pu'$,  then 
$$
u'v\equiv -x^2\equiv (au'-v)(pu'-bv)\equiv u'v+bv^2\mod p
$$
which means that $p \mid v$ and then $p \mid x$, which contradicts the primitivity of $(u,v,x,y,z)$. Thus, $
\v_p(u)=0$, so $\left(p,\frac{u}{au-pv}\right)_p=\left(\frac{u(au-pv)}{p}\right)=\left(\frac{a}{p}\right)$. This proves the proposition 
in the case $v=p$.
\end{proof}

We can now compute the Brauer--Manin obstruction coming from $\cuA$.

\begin{thm}\label{thm_ex1}
Consider the surface $Y_{p,a,b}$. Note that from $ab=p-1$ and $p\equiv 1\mod 4$ we conclude $(\frac{a}{p})=(\frac{b}{p})$. 
Also, $(\frac{a}{p})=(\frac{b}{p})=-1$ holds if and only if $p\equiv 5\mod 8$, and $a$ and $b$ are both even.
\begin{enumerate}
\item[(a)] If $(\frac{a}{p})=(\frac{b}{p})=-1$ then $Y_{p,a,b}(\Q)=\emptyset$ and this failure of the Hasse principle is 
explained by the Brauer--Manin obstruction coming from $\cuA$.
\item[(b)] Otherwise, we have
\[
Y_{p,a,b}(\mathbb A_\Q)^{\Br X} \ne \emptyset.
\]
\end{enumerate}
\end{thm}

\begin{proof}
The first observation comes from the following facts: for $p\equiv 1\mod 4$ we have $(\frac{-1}{p})=1$ and $(\frac{q}{p})=1$ for odd prime divisors $q$ of $p-1$. Also, $(\frac{2}{p})=1$ if $p\equiv 1\mod 8$ and $(\frac{2}{p})=-1$ if $p\equiv 5\mod 8$.

If $(\frac{a}{p})=(\frac{b}{p})=-1$ then Proposition \ref{pro:ex1_inv_A} implies that $\sum_v\inv_v\cuA(s_v)=\frac{1}{2}$ for all $(s_v)_v\in Y_{p,a,b}(\A_\Q)$, so $Y_{p,a,b}(\A_\Q)^\cuA=\emptyset$.

If $(\frac{a}{p})=(\frac{b}{p})=1$ then Proposition \ref{pro:ex1_inv_A} and Proposition \ref{pro:ex1_inv_B} imply that 
$\sum_v\inv_v\cuA(s_v)=0$ for all ${(s_v)_v\in Y_{p,a,b}(\A_\Q)}$, and invariant maps of $\cuB$ and $\cuC$ at $p$ are 
surjective. Thus, $Y_{p,a,b}(\A_\Q)^{\Br}\neq\emptyset$.
\end{proof}

\begin{exa}
Consider a surface $Y_{13,2,6}$ given by
$$
\begin{cases}
y^2-13x^2=uv;\\
z^2-13x^2=(2u-13v)(u-6v)
\end{cases}
$$
Theorem \ref{thm_ex1} shows that $Y_{13,2,6}(\Q)=\emptyset$.
\end{exa}

\begin{exa}
Consider a surface $Y_{13,1,12}$ given by
$$
\begin{cases}
y^2-13x^2=uv;\\
z^2-13x^2=(u-13v)(u-12v)
\end{cases}
$$
Theorem \ref{thm_ex1} shows that there is no Brauer--Manin obstruction for $Y_{13,1,12}$ to the Hasse principle, and it has a trivial rational point $(1:0:0:0:1)\in Y_{13,1,12}(\Q)$.
\end{exa}

\begin{exa}
Consider a surface $Y_{13,12,1}$ given by
$$
\begin{cases}
y^2-13x^2=uv;\\
z^2-13x^2=(12u-13v)(u-v)
\end{cases}
$$
Theorem \ref{thm_ex1} shows that there is no Brauer--Manin obstruction for $Y_{13,12,1}$, and it has a rational point $(1:-3:2:7:16)\in Y_{13,12,1}(\Q)$ which is actually non-trivial to find.
\end{exa}

\begin{rema}
Note that we can change $a$ and $b$ up to squares without changing the isomorphism class of the surface. Unless $(a,b)$ equals either $(1,p-1)$ or $(-1,-(p-1))$, up to squares, there does not seem to be an obvious rational point. However, for each explicit surface we studied we were always able to find such a point.

The machinery from \cite{CTSSD} and \cite{WittenbergBook} seems ideally suited to proving such a statement by making $X$ into an elliptic fibration using the projection $f \colon X \dashrightarrow \mathbb P^2$ away from a plane studied in the beginning of this paper. Unfortunately the ``condition (D)'' in those works is not satisfied, which relates to the fact that the Brauer group of $X$ is not vertical.
\end{rema}

\begin{rema}
Also, note that the technique described above would not just establish the existence of a rational point, but even prove that these are Zariski dense. We already know these two statements to be equivalent by work of Salberger and Skorobogatov \cite{SalbergerSkorobogatov91}.
\end{rema}

\subsection{Obstructions coming from $\mathcal B$}

We will now consider a second subfamily for which $\mathcal B$ might give an obstruction.

\begin{defi}\label{def_ex_B}
Let $p$ be an odd prime, and $a$ and $b$ integers such that
\begin{enumerate}
\item $(a+b-1)^2-4ab=p$,
\item $4a\equiv 4b\equiv 1\mod p$,
\item $a\equiv b\equiv 1\mod 2$, and
\item $a\equiv 1\mod 8$.
\end{enumerate}
Let $S_{p,a,b}$ be the surface given by
\begin{equation}\label{sys_ex_B}
\begin{cases}
y^2-px^2=uv;\\
z^2-px^2=(u+v)(au+bv).
\end{cases}
\end{equation}
\end{defi}

Such triples $(p,a,b)$ do exist, as one can see from considering $p \equiv 5 \mod 8$, an integer $t \equiv 3\frac{p-1}{4}\mod 8$ and
\[
a_t=t^2p^2-tp-\frac{p-1}{4} \quad \text{ and } \quad  b_t=t^2p^2+tp-\frac{p-1}{4}.
\]

One can also check that $S_{p,a,b}=X_{p,1,1,a,b,1,1}$ satisfies conditions (C1) and (C2).

\begin{pro}
\
\begin{enumerate}
\item[(a)] All $S_{p,a,b}$ are everywhere locally soluble.
\item[(b)] For all $S_{p,a,b}$ the invariant map of $\cuA$ at $p$ is surjective.
\end{enumerate}
\end{pro}

\begin{proof}
Proposition \ref{pro:loc_sol_subfam} addresses the local solubility at all places $v\not\in\{2,p\}$. Case 3 in the proof of Proposition~\ref{prop:surjectiveinvariantmap} establishes the local solubility at $p$ and the surjectivity of the invariant map of $\cuA$ at $p$. The local solubility at $2$ follows from the point $(1:0:0:0:\sqrt{a})\in X(\Q_2)$, since $a\equiv 1\mod 8$.
\end{proof}

We will now consider the obstruction coming from $\cuB$.

\begin{thm}\label{thm_ex_B}
We have $S_{p,a,b}(\mathbb A_\Q)^\cuB=\emptyset$ and in particular $S_{p,a,b}(\Q)=\emptyset$.
\end{thm}

\begin{proof}
We can prove that $\inv_v$ is identically zero for $v \ne p$ as we did in the proof of Proposition~\ref{pro:ex1_inv_A}.

Finally, consider the case $v=p$. Let $s_p=(u:v:x:y:z)$ such that $(u,v,x,y,z)$ is primitive in $\Z_p$, hence in particular $\v_p(u)\v_p(v)=0$. Assume $\v_p(u)=0$; the other case is similar. The system \eqref{sys_ex_B} together with the condition $4a\equiv 4b\equiv 1\mod p$ gives us
$$
v\equiv \frac{y^2}{u}\mod p\quad \text{ and } \quad z^2\equiv \frac{1}{4}(u+v)^2\mod p.
$$
Let us write $z \equiv \frac{\ep}2(u+v) \mod p$ with $\ep = \pm 1$. Then we get
\begin{align*}
z\pm y \equiv \frac{\ep}2(u+v) \pm y  & \equiv \frac{\ep}2 \frac{y^2}u \pm y + \frac{\ep}2u\\
 & \equiv \frac{\ep}2u\left(\frac{y^2}{u^2} \pm 2 \ep \frac yu + 1 \right)\\
& \equiv \frac{\ep}2u\left(\frac yu \pm \ep\right)^2 \mod p.
\end{align*}

So at least one of $\eta = \frac u{z\pm y}$ is non-zero modulo $p$ at $s_p$, and we have $(p,\eta)=-1$ since $(\frac{2}{p})=-1$. We conclude $\inv_p\cuB(s_p)=\frac{1}{2}$.
\end{proof}

\begin{exa}
Consider $S_{13,153,179}$ given by
$$
\begin{cases}
y^2-13x^2=uv;\\
z^2-13x^2=(u+v)(153u+179v).
\end{cases}
$$
Theorem \ref{thm_ex_B} shows that $S_{13,153,179}(\Q)=\emptyset$ and hence this surface does not satisfy the Hasse principle.
\end{exa}

\end{document}